\documentclass[11pt]{article}
\usepackage{mathrsfs}
\usepackage{amsfonts}
\usepackage{amsthm,amsmath,amssymb,anysize}
\newtheorem{lemma}{Lemma}[section]
\newtheorem{theorem}[lemma]{Theorem}

\newtheorem{claim}{\ \ Claim}
\newtheorem{proposition}[lemma]{Proposition}
\newtheorem{definition}[lemma]{Definition}

\usepackage[numbers,sort&compress]{natbib}
\marginsize{3.6cm}{3.6cm}{3.6cm}{3cm}
\numberwithin{equation}{section}

\title{\textsf{ Biderivations and commutative post-Lie algebra structures on the Lie algebra $\mathcal{W}(a,b)$}}
\author{
\textsc{Xiaomin Tang\footnote{E-mail: x.m.tang@163.com}}    \\
  \textit{Department of Mathematics},  \textit{Heilongjiang University}\\
  \textit{Harbin 150080, China}
     }
\date{ }
\begin{document}
\maketitle
\begin{quotation}
\noindent\textbf{Abstract.} For $a,b\in \mathbb{C}$, the Lie algebra $\mathcal{W}(a,b)$ is the semidirect product  of the Witt algebra and a module of the intermediate series. In this paper, all biderivations  of $\mathcal{W}(a,b)$ are determined. Surprisingly, these Lie algebras have symmetric (and skewsymmetric)  non-inner biderivations. As an applications, commutative post-Lie algebra structures on $\mathcal{W}(a,b)$ are obtained.

\vskip 5pt

 \noindent \textbf{Keywords}:   biderivation, Lie algebra $\mathcal{W}(a,b)$,  post-Lie algebra

\noindent \textbf{MSC 2010}: 17B05, 17B40, 17B68

  \end{quotation}

  \setcounter{section}{0}
\section{Introduction}\label{intro}

Derivations and generalized derivations (including biderivations) have become more and more powerful tools  in the structure  study of rings and  algebras. Besides their own interests, they have wide applications to other related problems. Recently, there are many efforts on this, see  \cite{Bre1995,Chen2016,Gho2013,Hanw,tang2016,WD1,WD3,WD2}. In his remarkable paper \cite{Bre1995}, Bre\v{s}ar  showed that all biderivations on commutative prime rings are inner biderivations, and determined the biderivations of semiprime rings. The notion of biderivations  was introduced to  Lie algebras in \cite{WD3}. Later   super-biderivations on some super-algebras was introduced in \cite{fan2016,WD2}.  For the last few years many authors computed  only  skew-symmetric biderivations of some Lie (super)algebras due to their close relation to commuting maps, see \cite{Chen2016,fan2016,Hanw,WD1,WD3,WD2}. Non-skew-symmetric biderivations should not be ignored. Actually non-skew-symmetric biderivations can be used to study  post-Lie algebras structures on Lie algebras.
This is addressed only quite recently. For example,  all biderivations of finite-dimensional complex simple Lie algebras,  all biderivations of some W-algebras,  all  biderivations of the twisted Heisenberg-Virasoro algebra,  all  biderivations of Block algebras,  and all the super-biderivations of classical simple Lie superalgebras  were given in
 \cite{tang2016, tang2017, tangli2017, Liuxw, yuant2017} respectively.

  The present paper is to find efficient ways to determine all  biderivations of the Lie algebras $\mathcal{W}(a,b)$ for all $a,b\in\mathbb{C}$, to recover and    generalize results in the papers \cite{Hanw, tang2017,tangli2017}. The Lie algebras $\mathcal{W}(a,b)$ is a class of interesting ones including  many important  Lie algebras as special cases. Let us first recall the Lie algebras $\mathcal{W}(a,b)$.

Throughout the paper, we denote by  $\mathbb{C}$ and $\mathbb{Z}$   the sets of complex numbers and integers, respectively. All vector spaces and algebras are over $\mathbb{C}$. For  $a,b\in \mathbb{C}$, the Lie  algebra $\mathcal{W}(a,b)={\rm span}_{\mathbb{C}}\{L_m, I_m|m\in \mathbb{Z}\}$ has the following brackets:
\begin{eqnarray*}
[L_m,L_n ]=(m-n)L_{m+n}, \ [L_m,I_n]=-(n+a+bm)I_{m+n},  \ [I_m,I_n]=0
\end{eqnarray*}
for all $m,n\in \mathbb{Z}$.
Note that $\mathcal{W}(a, b)$ contains a subalgebra $\mathcal{W} ={ \rm span}_{\mathbb{C}}\{L_m|m\in \mathbb{Z}\}$
isomorphic to the well-known Witt algebra, and that the space ${ \rm span}_{\mathbb{C}}\{I_m|m\in \mathbb{Z}\}$ is  a $\mathcal{W}$-module of the intermediate series. The algebras $\mathcal{W}(a,b)$ were considered in
the mathematical physics \cite{Ovs}.   We know that the universal
central extension of  $\mathcal{W}(0,0)$ is   the so-called
twisted Heisenberg-Virasoro algebra \cite{Arb}, which plays an important role in
the representation theory of toroidal Lie algebras \cite{Billig}. The universal
central extension of $\mathcal{W}(0,-1)$ is  the  Lie algebra $W(2,2)$ whose
representations have been studied in \cite{Zhangw} in terms of vertex operator algebras.

In  \cite{Hanw}, the authors  determined  skew-symmetric biderivations for all  $\mathcal{W}(a,b)$. All biderivations of $\mathcal{W}(0, -1)$ and $\mathcal{W}(0, 0)$  were later obtained in  \cite{tang2017,tangli2017}.     In the present  paper, we shall use the methods in  \cite{Liuxw}  to determine all  biderivations of  $\mathcal{W}(a,b)$   for all  $a,b\in\mathbb{C}$.

 The paper is organizes  as follows. In Section 2, we give general results on biderivations and some lemmas which will be used to our proof. In Section \ref{sec3}, we complete characterize the biderivations without the skew-symmetric condition of the Lie algebra $\mathcal{W}(a,b)$ for all cases of $a,b$. In Section \ref{sec4}, by using the biderivations we characterize the forms of the commutative post-Lie algebra structures on  $\mathcal{W}(a,b)$.

\section{General results on biderivations and some lemmas}\label{sec2}

Let $L$ be a Lie algebra.  Recall that a linear map $\phi: L\rightarrow L$ is called a derivation if   $\phi([x,y])=[\phi(x),y]+[x,\phi(y)]$
for all $x, y\in L$. For any $x\in L$, we have the inner derivation
$${\rm ad} x: L\rightarrow L, y\mapsto {\rm ad} x(y)=[x,y], \forall y\in L.$$    Denote by ${\rm{Der}} (L)$ and by ${\rm{Inn}} (L)$ the space of all derivations and the space of all  inner derivations of $L$ respectively. Now let us recall the definition of  a biderivation of a Lie algebra as follows.

\begin{definition}
A bilinear map $f : L\times L \rightarrow L$ is called a
biderivation of $L$ if it is a derivation with respect to both components. Namely, for all $x,y,z\in L$,
\begin{eqnarray*}
f([x,y],z)=[x,f(y,z)]+[f(x,z),y], \label{2der}\\
f(x,[y,z])=[f(x,y),z]+[y,f(x,z)]. \label{1der}
\end{eqnarray*}
\end{definition}

For any $\lambda\in\mathbb{C}$,  the bilinear map $f: L\times L\rightarrow L$ given by $f(x,y)=\lambda [x,y]$ for all $x, y\in L$
is a biderivation of $L$. Such biderivations are said to be inner.
Denote by ${\rm Bid} (L)$ the set of all biderivations of $L$ which is clearly a vector space. An $f\in {\rm Bid} (L)$ is called symmetric if
$f(x,y) = f(y,x)$ for all $x,y\in L$, and is called skew-symmetric if $f(x,y) = -f(y,x)$ for all$x,y\in L$. Denote by
${\rm Bid}_+(L)$ and ${\rm Bid}_-(L)$ the subspaces of all symmetric biderivations and all skew-symmetric
biderivations on $L$ respectively. If $f\in {\rm Bid} (L)$, then it is easy to see that the bilinear map $f^{\rm op}:L\times L\rightarrow L$ given by $f^{\rm op}(x,y)=f(y,x)$ for all $x,y\in L$ is also a biderivation of $L$. Let $f_-=\frac{1}{2}(f-f^{\rm op})$ and $f_+=\frac{1}{2}(f+f^{\rm op})$. It follows that $f_- \in {\rm Bid}_-(L)$ and $f_+\in {\rm Bid}_+(L)$ if $f\in {\rm Bid}(L)$. In view of $f=f_- + f_+$, the following result established in \cite{Liuxw} is very useful for our later  arguments.

\begin{lemma}\label{zhao}
 Let $L$ be any Lie algebra. Then
${\rm Bid}(L)={\rm Bid}_-(L) \oplus {\rm Bid}_+(L).$
\end{lemma}

In view of Lemma \ref{zhao}, to determine ${\rm Bid}(L)$ we only need to determine ${\rm Bid}_-(L)$ and ${\rm Bid}_+(L)$. The following lemmas are easy to verify by direct computations.

\begin{lemma}\label{auto}
Suppose that $L$ and $\widetilde{L}$ are two Lie algebras and $\sigma:L\rightarrow \widetilde{L}$ is an isomorphism of Lie algebras. For any bilinear map $f:L\times L\rightarrow L$, let the bilinear map $f^\sigma: \widetilde{L}\times \widetilde{L}\rightarrow \widetilde{L}$  be
determined by
$$
f^\sigma(\sigma(x),\sigma(y))=\sigma(f(x,y)) \text{ for all } x,y\in L.
$$
Then $f$ is a biderivation of $L$ if and only if $f^\sigma$ is a biderivation of $\widetilde{L}$.
\end{lemma}

\begin{lemma}\label{iso}
Let $k\in \mathbb{Z}$, $a, b\in\mathbb{C}$. Then the linear map $\sigma:\mathcal{W}(a,b)\rightarrow \mathcal{W}(a+k,b)$ given by  $\sigma(L_m)= L_m$, $\sigma(I_m)=I_{m-k}$ is an isomorphism of Lie algebras.
\end{lemma}

\begin{lemma}\cite{tang2017}\label{refe} Suppose that $k_i^{(n)},h_i^{(m)}\in\mathbb{C}$ satisfy
\begin{eqnarray*}
(i-m)k_i^{(n)}=(2n-m-i)h_{m-n+i}^{(m)}  \text{ for all } m,n,i\in\mathbb{Z}.
\end{eqnarray*}
Then there exists $\lambda\in\mathbb{C}$ such that $k_i^{(m)}=h_i^{(m)}=\delta_{m,i}\lambda$.
\end{lemma}

\section{ Biderivations of  $\mathcal{W}(a,b)$}\label{sec3}

 In this section, we shall determine all   biderivations of $\mathcal{W}(a,b)$. Thanks to Lemma 2.4, we may assume that $0\le a<1$.
 First we define three  classes of  biderivations for various $\mathcal{W}(a,b)$. The verifications are straightforward.
\begin{definition} \label{taa}
Let $\Omega=(\mu_k)_{  k\in \mathbb{Z}}$ be a sequence which contains only finitely many nonzero entries.

\begin{itemize}

\item The  biderivation $\Psi_\Omega: \mathcal{W}(a,0) \times \mathcal{W}(a,0)\rightarrow \mathcal{W}(a,0)$  is given by
\begin{eqnarray*}\label{fff35}\aligned
 \Psi_\Omega(L_m,L_n)=&\sum_{k \in \mathbb{Z}}\mu_k I_{m+n+k}, \\
\Psi_\Omega(L_m,I_n)=&\Psi_\Omega(I_n,L_m)=\Psi_\Omega(I_m,I_n)=0 \text{\ for all } m,n\in\mathbb{Z}.\endaligned
\end{eqnarray*}

\item The  biderivation $\Upsilon^a_\Omega: \mathcal{W}(a,1) \times \mathcal{W}(a,1)\rightarrow \mathcal{W}(a,1)$  is determined by
\begin{eqnarray*}\label{fff36}\aligned
 \Upsilon^a_\Omega(L_m,L_n)=&\sum_{k \in \mathbb{Z}}(m+n+k+a)\mu_k I_{m+n+k}, \\
\Upsilon^a_\Omega(L_m,I_n)=&\Upsilon^a_\Omega(I_n,L_m)=\Upsilon^a_\Omega(I_m,I_n)=0 \text{\ for all } m,n\in\mathbb{Z}.\endaligned
\end{eqnarray*}

\item The   biderivation $\Theta^a_\mu: \mathcal{W}(a,-1) \times \mathcal{W}(a,-1)\rightarrow \mathcal{W}(a,-1)$ for $a\in \mathbb{Z}$ is determined by
\begin{eqnarray*}\label{fff37}\aligned
 \Theta^a_\mu(L_m,L_n)&=(m-n)\mu I_{m+n-a}, \\
\Theta^a_\mu(L_m,I_n)&=\Theta^a_\mu(I_n,L_m)=\Theta^a_\mu(I_m,I_n)=0 \text{\ for all } m,n\in\mathbb{Z}.\endaligned
\end{eqnarray*}
\end{itemize}
\end{definition}

Note that $\Psi_\Omega,\Upsilon^a_\Omega$ are symmetric and $\Theta^a_\mu$ is skew-symmetric, and   they  are  non-inner if they are nonzero.  We know that  the spaces
${\rm Bid}(\mathcal{W}(0,-1))$, ${\rm Bid}(\mathcal{W}(0,0))$ and ${\rm Bid}_-(\mathcal{W}(a,b))$ were determined in
\cite{tang2017}, \cite{tangli2017} and \cite{Hanw}, respectively:

(1) If $f\in {\rm Bid}(\mathcal{W}(0,0))$, then there exist $\lambda \in \mathbb{C}$ and a sequence $\Omega=(\mu_k)_{  k\in \mathbb{Z}}$  which contains only finitely many nonzero entries such that $$f(x,y)=\lambda[x,y]+\Psi_\Omega(x,y) \text{\ for all } x,y\in \mathcal{W}(0,0);$$

(2) If $f\in {\rm Bid}(\mathcal{W}(0,-1))$, then  there exist $\lambda,\mu \in \mathbb{C}$ such that
$$f(x,y)=\lambda[x,y]+\Theta^0_\mu(x,y) \text{\ for all } x,y\in \mathcal{W}(0,-1);$$

(3) If $f\in {\rm Bid}_-(\mathcal{W}(a,b))$,  then there exist $\lambda,\mu \in \mathbb{C}$ such that
\begin{eqnarray*}
f(x,y)=
\begin{cases}
\lambda[x,y]+\Theta^a_\mu(x,y), & \text{if} \ a\in \mathbb{Z}, b=-1, \\
\lambda[x,y], & \text{otherwise}, \\
\end{cases}
\end{eqnarray*}
for all $x,y\in \mathcal{W}(a,b)$.

Now we present our main result in this section.

\begin{theorem}\label{posttheo}
 Any   biderivation  $f$ of $\mathcal{W}(a,b)$ is of the form
\begin{eqnarray*}f(x,y)=
\begin{cases}
\lambda[x,y]+\Psi_\Omega(x,y), & \text{if} \ b=0, \\
\lambda[x,y]+\Upsilon^a_\Omega(x,y), & \text{if} \ b=1, \\
\lambda[x,y]+\Theta^a_\mu(x,y), & \text{if} \ a\in \mathbb{Z}, b=-1, \\
\lambda[x,y], & \text{otherwise}. \\
\end{cases}
\end{eqnarray*}
for some  $\lambda,\mu \in \mathbb{C}$ and a sequence $\Omega=(\mu_k)_{  k\in \mathbb{Z}}$  which contains only finitely many nonzero entries.
\end{theorem}

The proof of Theorem \ref{posttheo} will be completed later. We first recall and establish  several auxiliary results.

\begin{lemma}\label{innerLie} \cite{Shoulan}
The derivation of $\mathcal{W}(a,b)$ is determined by the following:
\begin{eqnarray*}
{\rm{Der}} (\mathcal{W}(a,b))&=&\left\{
\begin{array}{lll}
{\rm{Inn}} (\mathcal{W}(a,b)) \oplus \mathbb{C} D_{1}\oplus \mathbb{C} D_{2}^{0,0} \oplus \mathbb{C} D_3, & (a,b)=(0,0);\\
{\rm{Inn}} (\mathcal{W}(a,b)) \oplus \mathbb{C} D_{1}\oplus \mathbb{C} D_{2}^{0,1}, & (a,b)=(0,1);\\
{\rm{Inn}} (\mathcal{W}(a,b)) \oplus \mathbb{C} D_{1}\oplus \mathbb{C} D_{2}^{0,2}, & (a,b)=(0,2);\\
{\rm{Inn}} (\mathcal{W}(a,b)) \oplus \mathbb{C} D_{1}, & \text{otherwise},\\
\end {array}
\right.
\end{eqnarray*}
where the derivations $D_{1},
D_{2}^{0,0},
D_{2}^{0,1},
D_{2}^{0,2},
D_{3}$
are defined as follows
for all $m\in \mathbb{Z}$,
\begin{eqnarray*}
D_{1}(L_m)=0, \ D_{1}(I_m)=I_m, \\
D_{2}^{0,0}(L_m)=(m-1)I_m, \ D_{2}^{0,0}(I_m)=0, \\
D_{2}^{0,1}(L_m)=(m^2-m)I_m, \ D_{2}^{0,1}(I_m)=0, \\
D_{2}^{0,2}(L_m)=m^3I_m, \ D_{2}^{0,2}(I_m)=0, \\
D_{3}(L_m)=mI_m, \ D_{3}(I_m)=0.
\end{eqnarray*}
\end{lemma}

\begin{lemma}\label{referee}
Suppose that $f$ is a biderivation of $ \mathcal{W}(a,b)$ with $(a,b)\neq (0,0)$.  Then there are linear maps $\phi^{a,b}$ and $\psi^{a,b}$ from $ \mathcal{W}(a,b)$ into itself such that
\begin{eqnarray*}
f(x,y)&=&\rho_1^{a,b}(x) D_1(y)+\rho^{a,b}_2(x)D^{a,b}_2(y)+[\phi^{a,b}(x), y]\nonumber\\
&=&\theta^{a,b}_1(y) D_1(x)+\theta^{a,b}_2(y)D^{a,b}_2(x)+[x, \psi^{a,b}(y)] \label{abcd1} \label{aaa2}
\end{eqnarray*}
for all $x,y\in \mathcal{W}(2,2)$, where $\rho^{a,b}_1, \rho^{a,b}_2$ and $\theta^{a,b}_1,\theta^{a,b}_2$ are linear complex valued functions on $\mathcal{W}(a,b)$, and $D_1, D_2^{a,b}$ are given by Lemma \ref{innerLie}, note that $D_2^{a,b}=0$ when $(a,b)\not\in\{(0,1),(0,2)\}$.
\end{lemma}

\begin{proof}
It is easy to see that, for the biderivation $f$ of $\mathcal{W}(a,b)$ and a fixed element $x\in \mathcal{W}(a,b)$, the linear map $\phi_x(y)=f(x,y)$ is a derivation of $\mathcal{W}(a,b)$. Notice that $(a,b)\neq (0,0)$, by Lemma \ref{innerLie}, there are complex-valued functions $\rho^{a,b}_1, \rho^{a,b}_2$ on $\mathcal{W}(a,b)$ and a linear map $\phi^{a,b}$ from $\mathcal{W}(a,b)$ into itself such that $\phi_x =\rho^{a,b}_1(x) D_1+\rho^{a,b}_2(x) D_2^{a,b}+{\rm ad}\phi^{a,b}(x)$, where we provide that $D_2^{a,b}=0$ when $(a,b)\not\in\{(0,1),(0,2)\}$.
Namely, $f(x,y)=\rho^{a,b}_1(x) D_1(y)+\rho^{a,b}_2(x) D_2^{a,b}(y)+[\phi^{a,b}(x), y]$. Because $f$ is bilinear, the maps $\rho^{a,b}_1, \rho^{a,b}_2$ are linear. Similarly, the map $\psi_z(y)=f(y,z)$ is a derivation of $\mathcal{W}(a,b)$, and there are linear complex valued functions $\theta^{a,b}_1,\theta^{a,b}_2$ on $\mathcal{W}(a,b)$ and a linear map $\psi^{a,b}$ from $\mathcal{W}(a,b)$ into itself such that
\begin{eqnarray*}f(x,y)&=&\theta^{a,b}_1(y) D_1(x)+\theta^{a,b}_2(y)D_2^{a,b}(x)+{\rm ad}(-\psi^{a,b} (y))(x)\nonumber\\
&=&\theta^{a,b}_1(y) D_1(x)+\theta^{a,b}_2(y)D_2^{a,b}(x)+[x, \psi^{a,b}(y)].\label{aaa3}
\end{eqnarray*}
 The proof is completed.
\end{proof}

\begin{lemma}\label{ref} Let $f$ be a biderivation of $\mathcal{W}(a,b)$ with $(a,b)\neq (0,0)$, and $\phi^{a,b}$, $\psi^{a,b}$, $\rho^{a,b}_i$, $\theta^{a,b}_i, i=1,2$ be given as in Lemma \ref{referee}. Then the following equations hold.
\begin{eqnarray}
 f(L_m,L_n)&=&\rho^{a,b}_2(L_m)D^{a,b}_2(L_n)+[\phi^{a,b} (L_m),L_n]\nonumber \\
&=& \theta^{a,b}_2(L_n)D^{a,b}_2(L_m)+[L_m,\psi^{a,b}(L_n)], \label{aa14}
\end{eqnarray}
\begin{eqnarray}
f(L_m,I_n)&=&\rho^{a,b}_1(L_m)I_n+[\phi^{a,b} (L_m),I_n] \nonumber \\
&=& \theta^{a,b}_2(I_n)D^{a,b}_2(L_m) +[L_m,\psi^{a,b} (I_n)], \label{aa15}
\end{eqnarray}
\begin{eqnarray}
f(I_n,L_m)&=&\rho^{a,b}_2(I_n)D^{a,b}_2(L_m)+[\phi^{a,b} (I_n),L_m] \nonumber \\
&=&\theta^{a,b}_1(L_m)I_n+[I_n,\psi^{a,b}(L_m)],\label{aa16}
\end{eqnarray}
\begin{eqnarray}
f(I_m,I_n)&=&\rho^{a,b}_1(I_m)I_n+[\phi^{a,b} (I_m),I_n] \nonumber \\
&=&\theta^{a,b}_1(I_n)I_m+[I_m,\psi^{a,b}(I_n)]. \label{aa17}
\end{eqnarray}
\end{lemma}

\begin{proof}
It will follow by Lemmas \ref{innerLie} and \ref{referee}.
\end{proof}

Let $f$ be a biderivation of $\mathcal{W}(a,b)$. In view of Lemma \ref{referee}, we can assume that
\begin{eqnarray}
\phi^{a,b}(L_n)=\sum_{i\in \mathbb{Z}}a_i^{a,b}{(n)}L_i+\sum_{i\in \mathbb{Z}}b_i^{a,b}{(n)}I_i,\label{ll1}   \\
\psi^{a,b}(L_n)=\sum_{i\in \mathbb{Z}}c_i^{a,b}{(n)}L_i+\sum_{i\in \mathbb{Z}}d_i^{a,b}{(n)}I_i, \label{ll2}\\
\phi^{a,b}(I_n)=\sum_{i\in \mathbb{Z}}p_i^{a,b}{(n)}L_i+\sum_{i\in \mathbb{Z}}q_i^{a,b}{(n)}I_i, \label{dd10}\\
\psi^{a,b} (I_n)=\sum_{i\in \mathbb{Z}}s_i^{a,b}{(n)}L_i+\sum_{i\in \mathbb{Z}}r_i^{a,b}{(n)}I_i, \label{eee1}
\end{eqnarray}
where $a_i^{a,b}{(n)}, b_i^{a,b}{(n)}, c_i^{a,b}{(n)}, d_i^{a,b}{(n)},p_i^{a,b}{(n)}, q_i^{a,b}{(n)}, r_i^{a,b}{(n)}, s_i^{a,b}{(n)}\in \mathbb{C}$.

According to Lemma \ref{zhao} and \cite{Hanw}, we only need to determine the sets ${\rm Bid}_+(\mathcal{W}(a,b))$. That is, we
have to study the symmetric biderivations $of \mathcal{W}(a,b)$. It will be  divided into several cases based on the values of $a,b$.
For $f\in {\rm Bid} (\mathcal{W}(a,b))$, the notions given in (\ref{ll1})-(\ref{eee1}) will be applied always.

\subsection{The case for $(a,b)=(0,1)$}
\

\ \ \ \ By the definition, $\mathcal{W}(0,1)$ has the following Lie brackets
\begin{eqnarray*}
[L_m,L_n ]=(m-n)L_{m+n},\ [L_m,I_n ]=-(m+n)I_{m+n},\  [I_m,I_n ]=0.
\end{eqnarray*}

In this case, we shall prove the following result.

\begin{proposition}\label{pro1}
Let $f\in \rm Bid_+(\mathcal{W}(0,1))$. Then there is a sequence $\Omega=(\mu_k)_{  k\in \mathbb{Z}}$  which contains only finitely many nonzero entries such that
$
f(x,y)=\Upsilon_\Omega^{0}(x,y)
$
for all $x,y\in \mathcal{W}(0,1)$.
\end{proposition}

\begin{proof}
The proof will be completed by verifying the following three claims.

\begin{claim}\label{claim 1}
There is a sequence $\Omega=(\mu_k)_{  k\in \mathbb{Z}}$  which contains only finitely many nonzero entries such that
$$
f(L_m,L_n)=\sum_{k\in \mathbb{Z}}(m+n+k)\mu_k I_{m+n+k}, \text{ for all } m,n\in \mathbb{Z}.
$$
\end{claim}
By Lemma \ref{referee}, the linear maps $\phi^{0,1}, \psi^{0,1}: \mathcal{W}(0,1)\rightarrow \mathcal{W}(0,1)$ satisfy (\ref{aa14}).
This, together with (\ref{ll1}) and (\ref{ll2}), yields that $f(L_m,L_n)$ is equal to
\begin{eqnarray}
&&\rho^{0,1}_2(L_m)(n^2-n)I_n+\sum_{i\in \mathbb{Z}}(i-n)a_i^{0,1}{(m)}L_{n+i}+\sum_{i\in \mathbb{Z}}(i+n)b_i^{0,1}{(m)}I_{n+i} \nonumber\\
 &=& \theta^{0,1}_2(L_n)(m^2-m)I_m+ \sum_{j\in \mathbb{Z}}(m-j)c_j^{0,1}{(n)}L_{m+j}-\sum_{j\in \mathbb{Z}}(m+j)d_j^{0,1}{(n)}I_{m+j}\nonumber\\
 &=& \theta^{0,1}_2(L_n)(m^2-m)I_m+ \sum_{i\in \mathbb{Z}}(2m-n-i)c_{n-m+i}^{0,1}{(n)}L_{n+i}  \label{abab1}\\
 && \ \ \ \ -\sum_{i\in \mathbb{Z}}(i+n)d_{n-m+i}^{0,1}{(n)}I_{n+i}. \nonumber
\end{eqnarray}
Because $f$ is symmetric which implies  $f(L_m, L_n)= f(L_n, L_m)$, comparing both sides of the above equations, we obtain
\begin{equation}\label{ac111}
(i-n)a_i^{0,1}{(m)}=(n-i)c_i^{0,1}{(m)}=(2m-n-i)c_{n-m+i}^{0,1}{(n)}, \ \text{ for all } m,n,i\in \mathbb{Z}
\end{equation}
and for any $m\neq n$ with $i\neq 0, m-n$,
\begin{eqnarray}
(i+n)b_i^{0,1}{(m)}=-(i+n)d_i^{0,1}{(m)}=-(i+n)d_{n-m+i}^{0,1}{(n)}, \label{ef11}\\
\rho^{0,1}_2(L_m)(n^2-n)+nb_{0}^{0,1}{(m)}=-nd_{n-m}^{0,1}{(n)},\label{ef12} \\
\theta^{0,1}_2(L_n)(m^2-m)-md_{0}^{0,1}{(n)}=mb_{m-n}^{0,1}{(m)}. \label{ef13}
\end{eqnarray}
By Lemma \ref{refe} with (\ref{ac111}) and (\ref{ef11}), one has $a_i^{0,1}{(m)}=c_i^{0,1}{(m)}=0$ and
\begin{equation}\label{ac211}
b_i^{0,1}{(m)}=-d_i^{0,1}{(m)}, \ d_j^{0,1}{(m)}=d_{n-m+j}^{0,1}{(n)}\ \text{ for all } i\neq 0, j\neq 0, -n, m-n.
\end{equation}
Denote $J_1=\{-2,1,3\}$, $J_2=\{-3,3,0\}$ and $J_3=\{-3,1,4\}$. In view of (\ref{ac211}), we have
\begin{equation}\label{abcd22}
\begin{cases}
 d_{-2-m}^{0,1}{(-2)}=d_{3-m}^{0,1}{(3)}, & \text{ if } m\notin J_1,  \\
 d_{3-m}^{0,1}{(3)}=d_{-3-m}^{0,1}{(-3)},  & \text{ if } m\notin J_2, \\
 d_{-3-m}^{(-3)}=d_{4-m}^{(4)},   & \text{ if } m\notin J_3.
 \end{cases}
\end{equation}
Take $n=-2,3,-3$ and $4$ in (\ref{ef12}), respectively, we obtain
\begin{eqnarray}
6\rho^{0,1}_2(L_m)-2b_0^{0,1}{(m)}=2d_{-2-m}^{0,1}{(-2)}, & \text{if } m\neq -2,\label{ef15} \\
6\rho^{0,1}_2(L_m)+3b_0^{0,1}{(m)}=-3d_{3-m}^{0,1}{(3)}, &\text{if }m\neq 3,\label{ef16} \\
12\rho^{0,1}_2(L_m)-3b_0^{0,1}{(m)}=3d_{-3-m}^{0,1}{(-3)}, &\text{if }m\neq -3,\label{ef17} \\
12\rho^{0,1}_2(L_m)+4b_0^{0,1}{(m)}=-4d_{4-m}^{0,1}{(4)}, &\text{if }m\neq 4.\label{ef18}
\end{eqnarray}
According to (\ref{abcd22}), it follows by (\ref{ef15}) and (\ref{ef16}) that $\rho^{0,1}_2(L_m)=0$ if $m\notin J_1$;  by (\ref{ef16}) and (\ref{ef17}) that $\rho^{0,1}_2(L_m)=0$ if $m\notin J_2$ and by (\ref{ef17}) and (\ref{ef18}) that $\rho^{0,1}_2(L_m)=0$ if $m\notin J_3$. Note that $J_1\cap J_2\cap J_3=\emptyset$, we get $\rho^{0,1}_2(L_m)=0$ for all $m\in \mathbb{Z}$. Similarly, according to (\ref{ef13}) we have $\theta^{0,1}_2(L_n)=0$ for all $n\in \mathbb{Z}$.  This, together with (\ref{abab1}), implies that for any $m,n,i\in \mathbb{Z}$,
\begin{eqnarray}
(i+n)b_i^{0,1}{(m)}=-(i+n)d_i^{0,1}{(m)}=-(i+n)d_{n-m+i}^{0,1}{(n)}. \label{ef111}
\end{eqnarray}
It is not difficult to see by (\ref{ef111}) that $d_{m+k}^{0,1}{(m)}=d_{n+k}^{0,1}{(n)}=-b_{m+k}^{0,1}{(m)}$ for all $m,n,k\in \mathbb{Z}$. Denote $d_{k}^{0,1}{(0)}=-\mu_k$, we have $d_{m+k}^{0,1}{(m)}=-b_{m+k}^{0,1}{(m)}=-\mu_k$ for all $m,k\in \mathbb{Z}$. All these with (\ref{abab1}) yield that
\begin{eqnarray*}
f(L_m,L_n)&=&\sum_{i\in \mathbb{Z}}(i+n)b_i^{0,1}{(m)}I_{n+i}\\
&=&\sum_{k\in \mathbb{Z}}(m+k+n)b_{m+k}^{0,1}{(m)}I_{n+m+k}\\
&=&\sum_{k\in \mathbb{Z}}(m+k+n)\mu_k I_{n+m+k}.
\end{eqnarray*}
Note that the sequence $\Omega \doteq(\mu_k)_{  k\in \mathbb{Z}}$  which contains only finitely many nonzero entries for which the above equation makes sense.
This completes the proof of the claim. In addition, the above proof also implies
\begin{equation} \label{ps=0}
\phi^{0,1}(L_n)=\psi^{0,1}(L_n)=\sum_{k\in \mathbb{Z}} \mu_k I_{n+k}, \ \text{ for all } n\in \mathbb{Z}.
\end{equation}

\begin{claim}\label{claim 2}
$f(L_m,I_n)=0$ for all $m,n\in \mathbb{Z}$.
\end{claim}
By Lemma \ref{referee}, the linear maps $\phi^{0,1}, \psi^{0,1}: \mathcal{W}(0,1)\rightarrow \mathcal{W}(0,1)$ satisy (\ref{aa15}) and (\ref{aa16}).
This, together with (\ref{dd10}), (\ref{eee1}) and (\ref{ps=0}), yields  that $f(L_m,I_n)$ and $f(I_n, L_m)$ are of the following forms respectively:
\begin{eqnarray}
\rho^{0,1}_1(L_m)I_n&=& \theta^{0,1}_2(I_n)(m^2-m)I_m+\sum_{j\in \mathbb{Z}}(m-i)s_i^{0,1}{(n)}L_{m+i}\nonumber\\
&&\ \ \ \ -\sum_{i\in \mathbb{Z}}(m+i)r_i^{0,1}{(n)}I_{m+i},\label{abc1} \\
\theta^{0,1}_1(L_m)I_n&=&\rho^{a,b}_2(I_n)(m^2-m)I_m+\sum_{j\in \mathbb{Z}}(i-m)p_i^{0,1}{(n)}L_{i+m} \nonumber\\
&&\ \ \ +\sum_{i\in \mathbb{Z}}(m+i)q_i^{0,1}{(n)}I_{m+i}. \label{abc2}
\end{eqnarray}
For any $m\neq n$, from (\ref{abc1}) we get
\begin{eqnarray}
&& (m-i)s_i^{0,1}{(n)}=0, \label{170918-1}\\
&& nr_{n-m}^{0,1}{(n)}=\rho^{0,1}_1(L_m), \label{170918-2} \\
&& \theta^{0,1}_2(I_n)(m^2-m)-mr_0^{0,1}{(n)}=0. \label{170918-3}
\end{eqnarray}
It is easy to see by (\ref{170918-1}) that $s_i^{0,1}{(n)}=0$ for all $i,n\in \mathbb{Z}$.  Take $m=2,3,4$ and $5$ in (\ref{170918-3}), respectively, then we have
\begin{eqnarray}
2\theta^{0,1}_2(I_n)-2r_0^{0,1}{(n)}=0, \ \text{if }n\neq 2, \label{nneq2} \\
6\theta^{0,1}_2(I_n)-3r_0^{0,1}{(n)}=0, \ \text{if }n\neq 3, \label{nneq3}\\
12\theta^{0,1}_2(I_n)-4r_0^{0,1}{(n)}=0, \ \text{if }n\neq 4,\label{nneq4} \\
20\theta^{0,1}_2(I_n)-5r_0^{0,1}{(n)}=0, \ \text{if }n\neq 5.\label{nneq5}
\end{eqnarray}
It follows by (\ref{nneq2}) and (\ref{nneq3}) that $\theta^{0,1}_2(I_n)=0$ if $n\notin \{2,3\}$, and by (\ref{nneq4}) and (\ref{nneq5}) that $\theta^{0,1}_2(I_n)=0$ if $n\notin \{4,5\}$. In view of $\{2,3\}\cap \{4,5\}=\emptyset$, we obtain $\theta^{0,1}_2(I_n)=0$ for all $n\in \mathbb{Z}$.
In addition, by letting $n=0$ in (\ref{170918-2}) we obtain $\rho^{0,1}_1(L_m)=0$ for all $m\neq 0$. But we also have $\rho^{0,1}_1(L_0)=0$ by taking
$m=n=0$ in (\ref{abc1}). Now, it is already shown that $f(L_m,I_n)=0$ for all $m,n\in \mathbb{Z}$, which proves the claim.

On other hand, the results above
together with (\ref{abc1}) yield $(m+i)r_i^{0,1}{(n)}=0$ for any integers $m,n,i$. This implies $r_i^{0,1}{(n)}=0$ for all $n,i\in \mathbb{Z}$. Similarly, by (\ref{abc2}) we have $p_i^{0,1}{(n)}=q_i^{0,1}{(n)}=0$ for all $i,n\in \mathbb{Z}$. Hence, we get the following useful result:
\begin{equation} \label{pss=0}
\phi^{0,1}(I_n)=\psi^{0,1}(I_n)=0, \ \text{ for all } n\in \mathbb{Z}.
\end{equation}

\begin{claim}\label{claim 3}
$f(I_m,I_n)=0$ for all $m,n\in \mathbb{Z}$.
\end{claim}

By Lemma \ref{referee}, the linear maps $\phi^{0,1}, \psi^{0,1}: \mathcal{W}(0,1)\rightarrow \mathcal{W}(0,1)$ satisfy (\ref{aa17}).
This, together with (\ref{pss=0}), yields
\begin{eqnarray*}
f(I_m,I_n)=\rho^{0,1}_1(I_m)I_n=\theta^{0,1}_1(I_n)I_m.
\end{eqnarray*}
Let $m,n$ run all integers with $m\neq n$ in the above equation, then $\rho^{0,1}_1(I_m)=\theta^{0,1}_1(I_n)=0$ and the conclusion is proved.

Finally, the proof of the proposition is completed by Claims \ref{claim 1}, \ref{claim 2} and \ref{claim 3}.
\end{proof}

\subsection{The case for $(a,b)=(0,2)$}
\

\ \ \ \ By the definition, $\mathcal{W}(0,2)$ has the following Lie brackets
\begin{eqnarray*}
[L_m,L_n ]=(m-n)L_{m+n}, [L_m,I_n ]=-(2m+n)I_{m+n},  [I_m,I_n ]=0.
\end{eqnarray*}

In this case, we shall prove the following result.

\begin{proposition}\label{pro2}
Let $f\in \rm Bid_+(\mathcal{W}(0,2))$. Then $f(x,y)=0$ for all $x,y\in \mathcal{W}(0,2)$.
\end{proposition}

\begin{proof}
We shall complete the proof by verifying the following three claims.

\begin{claim}\label{claim 4}
$f(L_m,L_n)=0$ for all $m,n\in \mathbb{Z}$.
\end{claim}
By Lemma \ref{referee}, there are linear maps $\phi^{0,2}, \psi^{0,2}: \mathcal{W}(0,2)\rightarrow \mathcal{W}(0,2)$ satisfy (\ref{aa14}).
This, together with (\ref{ll1}) and (\ref{ll2}), yields that $f(L_m,L_n)$ is equal to
\begin{eqnarray}
&& \rho^{0,2}_2(L_m)n^3I_n+\sum_{i\in \mathbb{Z}}(i-n)a_i^{0,2}{(m)}L_{n+i}+\sum_{i\in \mathbb{Z}}(2n+i)b_i^{0,2}{(m)}I_{n+i} \nonumber\\
 &=& \theta^{0,2}_2(L_n)m^3I_m+ \sum_{j\in \mathbb{Z}}(m-j)c_j^{0,2}{(n)}L_{m+j}-\sum_{j\in \mathbb{Z}}(2m+j)d_j^{0,2}{(n)}I_{m+j} \nonumber\label{abab199}\\
 &=& \theta^{0,2}_2(L_n)m^3I_m+ \sum_{i\in \mathbb{Z}}(2m-n-i)c_{n-m+i}^{0,2}{(n)}L_{n+i} \\
 &&\ \ \ \ \ -\sum_{i\in \mathbb{Z}}(m+n+i)d_{n-m+i}^{0,2}{(n)}I_{n+i}. \nonumber
\end{eqnarray}
Because $f$ is symmetric which implies  $f(L_m, L_n)= f(L_n, L_m)$, comparing both sides of the above equations, we obtain
\begin{equation}\label{ac1}
(i-n)a_i^{0,2}{(m)}=(n-i)c_i^{0,2}{(m)}=(2m-n-i)c_{n-m+i}^{0,2}{(n)}, \ \text{ for all } m,n,i\in \mathbb{Z}
\end{equation}
and for $m\neq n$ and $i\neq 0, m-n$,
\begin{eqnarray}
(2n+i)b_i^{0,2}{(m)}=-(2n+i)d_i^{0,2}{(m)}=-(m+n+i)d_{n-m+i}^{0,2}{(n)}, \label{ef116}\\
\rho^{0,2}_2(L_m)n^3+2nb_{0}^{0,2}{(m)}=-(m+n)d_{n-m}^{0,2}{(n)},\label{ef126} \\
\theta^{0,2}_2(L_n)m^3-2md_{0}^{0,2}{(n)}=(m+n)b_{m-n}^{0,2}{(m)}. \label{ef136}
\end{eqnarray}
It is easy to see by Lemma \ref{refe} with (\ref{ac1}) and (\ref{ef116}) that $a_i^{0,2}{(m)}=c_i^{0,2}{(m)}=0$ and
\begin{equation}\label{ac2}
b_i^{0,2}{(m)}=-d_i^{0,2}{(m)}, \ \text{ for all } i\neq 0.
\end{equation}
By taking $i=-2n$ in (\ref{ef116}), we get $(n-m)d_{-m-n}^{0,2}{(n)}=0$ if $-2n\neq 0, m-n$. It follows that
\begin{equation}\label{ac26}
d_k^{0,2}{(n)}=0, \ \text{ for all } n\neq 0, k\neq 0, -2n.
\end{equation}
Let $m=0$ and $i=n\neq 0$ in (\ref{ef116}), then one has $d_n^{0,2}(0)=\frac{2}{3}d_{2n}^{0,2}(n)$. Note that $2n\neq 0, -2n$,  by
(\ref{ac26}) we have $d_{2n}^{0,2}(n)=0$, which yields $d_n^{0,2}(0)=0$ for all $n\neq 0$. By taking $n=0$ and $m\neq 0$ in (\ref{ef136}),
notice that $b_m^{0,2}(m)=-d_m^{0,2}(m)=0$ according to (\ref{ac2}) and (\ref{ac26}), we obtain $\theta^{0,2}_2(L_0)m^3-2md_{0}^{0,2}{(0)}=0$. This implies $d_{0}^{0,2}{(0)}=0$. Now, we have shown that
\begin{equation}\label{ac266}
b_s^{0,2}{(0)}=d_s^{0,2}{(0)}=b_k^{0,2}{(n)}=d_k^{0,2}{(n)}=0, \ \text{ for all } k,s,n\in \mathbb{Z} \ \text{with}\ k\neq 0, -2n.
\end{equation}
Then, from (\ref{ac266}) we know $b_{m-n}^{0,2}{(m)}=0$ if $m-n\neq 0, -2m$. This, together with (\ref{ef136}), deduces
$$
\theta^{0,2}_2(L_n)m^3-2md_{0}^{0,2}{(n)}=0, \text{ for all } n\neq m,3m.
$$
Take $m=1,2,4$ and $5$ in the above equation, respectively, we obtain
\begin{eqnarray}
\theta^{0,2}_2(L_n)-2d_0^{0,2}{(n)}=0, &\text{if } n\neq 1,3,\label{ef156} \\
8\theta^{0,2}_2(L_n)-4d_0^{0,2}{(n)}=0, &\text{if }n\neq 2,6,\label{ef166} \\
64\theta^{0,2}_2(L_n)-8d_0^{0,2}{(n)}=0, &\text{if }n\neq 4,12,\label{ef176} \\
125\theta^{0,2}_2(L_n)-10d_0^{0,2}{(n)}=0, &\text{if }n\neq 5,15.\label{ef186}
\end{eqnarray}
It follows by (\ref{ef156}) and (\ref{ef166}) that $\theta^{0,2}_2(L_n)=0$ if $n\notin \{1,2,3,6\}$,  by (\ref{ef176}) and (\ref{ef186}) that $\theta^{0,2}_2(L_n)=0$ if $n\notin \{4,5,12,15\}$. Note that $\{1,2,3,6\}\cap \{4,5,12,15\}=\emptyset$, we get $\theta^{0,2}_2(L_n)=0$ for all $n\in \mathbb{Z}$. Similarly, according to (\ref{ef126}) we have $\rho^{0,2}_2(L_m)=0$ for all $m\in \mathbb{Z}$.
This, together with (\ref{abab1}), implies that (\ref{ef116}) holds for any $m,n,i\in \mathbb{Z}$. By letting $i=0$ and $n=-m\neq 0$ in (\ref{ef116}),
we have $d_0^{0,2}(m)=0$ for all $m\neq 0$. This, together with (\ref{ac266}), gives that $b_k^{0,2}{(n)}=d_k^{0,2}{(n)}=0$ for all $k,n\in \mathbb{Z}$ with
$k\neq -2n$ if $n\neq 0$.
This completes the proof of the claim. In addition, the above proof also implies
\begin{equation} \label{ps=06}
\phi^{0,2}(L_n)=\psi^{0,2}(L_n)=0, \ \text{ for all } n\in \mathbb{Z}.
\end{equation}

\begin{claim}\label{claim 5}
$f(L_m,I_n)=0$ for all $m,n\in \mathbb{Z}$.
\end{claim}
By Lemma \ref{referee}, the linear maps $\phi^{0,2}, \psi^{0,2}: \mathcal{W}(0,2)\rightarrow \mathcal{W}(0,2)$ satisfy (\ref{aa15}) and (\ref{aa16}).
This, together with (\ref{ps=06}) , (\ref{dd10}) and (\ref{eee1}), yields that $f(L_m,I_n)$ and $f(I_n, L_m)$ are of the following forms respectively:
\begin{eqnarray}
\rho^{0,2}_1(L_m)I_n&=& \theta^{0,2}_2(I_n)m^3I_m+\sum_{i\in \mathbb{Z}}(m-i)s_i^{0,2}{(n)}L_{m+i}\nonumber \\
&&\ \ \ \ -\sum_{i\in \mathbb{Z}}(2m+i)r_i^{0,2}{(n)}I_{m+i},\label{abc16} \\
\theta^{0,2}_1(L_m)I_n&=&\rho^{0,2}_2(I_n)m^3I_m+\sum_{i\in \mathbb{Z}}(i-m)p_i^{0,2}{(n)}L_{m+i}\nonumber \\
&&\ \ \ \ +\sum_{i\in \mathbb{Z}}(2m+i)q_i^{0,2}{(n)}I_{m+i}.\label{abc26}
\end{eqnarray}
 From (\ref{abc16}) we get $(m-i)s_i^{0,2}{(n)}=0$ for all $m,n,i\in\mathbb{Z}$ which implies $s_i^{0,2}{(n)}=0$ for any integer $i,n\in \mathbb{Z}$,
 and by (\ref{abc16}) we also obtain for any $m\neq n$,
\begin{eqnarray}
&& (2m+i)r_i^{0,2}{(n)}=0,\ i\neq 0, n-m, \label{170918-16}\\
&& -(m+n)r_{n-m}^{0,2}{(n)}=\rho^{0,2}_1(L_m), \label{170918-26} \\
&& \theta^{0,2}_2(I_n)m^3-2mr_0^{0,2}{(n)}=0. \label{170918-36}
\end{eqnarray}
By letting $m=-n\neq n$ in (\ref{170918-26}), we have $\rho^{0,2}_1(L_m)=0$ for all $m\neq 0$. In addition, by taking $m=n=0$ in (\ref{abc16}) we have $\rho^{0,2}_1(L_0)=0$ and then $f(L_m,I_n)=\rho^{0,2}_1(L_m)I_n=0$ for all $m,n\in \mathbb{Z}$. This completes the proof of the claim.
On the other hand, in view of (\ref{170918-36}), in a similar way to Claim \ref{claim 4} we get $\theta^{0,2}_2(I_n)=r_0^{0,2}{(n)}=0$ for all $n\in \mathbb{Z}$.  These results with (\ref{170918-36}) yield $r_i^{0,2}{(n)}=0$ for all $i,n\in \mathbb{Z}$. Similarly, by (\ref{abc26}) we obtain
$\rho^{0,2}_2(I_n)=p_i^{0,2}{(n)}=q_i^{0,2}{(n)}=\theta^{0,2}_1(L_n)=0$ for any integer $n$. We have shown that
\begin{equation} \label{pss=06}
\phi^{0,2}(I_n)=\psi^{0,2}(I_n)=0, \ \text{ for all } n\in \mathbb{Z}.
\end{equation}

\begin{claim}\label{claim 6}
$f(I_m,I_n)=0$ for all $m,n\in \mathbb{Z}$.
\end{claim}

From (\ref{pss=06}), the proof can be finished by a similar way to Claim \ref{claim 3}.

Finally, the proof of the proposition is completed by Claims \ref{claim 4}, \ref{claim 5} and \ref{claim 6}.
\end{proof}

\subsection{The case for $a,b$ with $a\notin \mathbb{Z}$  or  $a=0$, $b\notin \{0, 1,2\}$}

In this case, we shall prove the following result.

\begin{proposition}\label{pro3}
Suppose that $a\notin \mathbb{Z}$  or  $a=0$, $b\notin \{0, 1,2\}$. Let $f\in {\rm Bit}_+(\mathcal{W}(a,b))$. Then there is a sequence $\Omega=(\mu_k)_{  k\in \mathbb{Z}}$  which contains only finitely many nonzero entries such that
$
f(x,y)=\Delta_\Omega(x,y)
$
for all $x,y\in \mathcal{W}(a,b)$, where $\Delta_\Omega$ is a bilinear map on $\mathcal{W}(a,b)$ defined by
$$
\Delta_\Omega(L_m,L_n)=
\begin{cases}
\sum_{k\in \mathbb{Z}}(b(m+n)+k+a)\mu_k I_{m+n+k}, & a\notin \mathbb{Z}, \ \ b=0 \ \text{or} \ 1;\\
0, &otherwise,
\end{cases}
$$
and $\Delta_\Omega(L_m,I_n)=\Delta_\Omega(I_m,L_n)=\Delta_\Omega(I_m,I_n)=0$ for all $m,n\in \mathbb{Z}$.
\end{proposition}

\begin{proof}
The proof will be completed by verifying the following three claims.
\begin{claim}\label{claim 7}
There is a sequence $\Omega=(\mu_k)_{  k\in \mathbb{Z}}$  which contains only finitely many nonzero entries such that
$$
f(L_m,L_n)=\begin{cases}
\sum_{k\in \mathbb{Z}}(b(m+n)+k+a)\mu_k I_{m+n+k}, & a\notin \mathbb{Z}, \ \ b=0 \ \text{or} \ 1;\\
0, &otherwise.
\end{cases}.
$$
\end{claim}
By Lemma \ref{referee}, the linear maps $\phi^{a,b}, \psi^{a,b}: \mathcal{W}(a,b)\rightarrow \mathcal{W}(a,b)$ satisfy (\ref{aa14}).
This, together with (\ref{ll1}) and (\ref{ll2}), yields that $f(L_m, L_n)$ is equal to
\begin{eqnarray}
& &\sum_{i\in \mathbb{Z}}(i-n)a_i^{a,b}{(m)}L_{n+i}+\sum_{i\in \mathbb{Z}}(i+a+bn)b_i^{a,b}{(m)}I_{n+i}\\
&=&\sum_{j\in \mathbb{Z}}(m-j)c_j^{a,b}{(n)}L_{m+j}-\sum_{j\in \mathbb{Z}}(j+a+bm)d_j^{a,b}{(n)}I_{m+j} \nonumber \\
&=&\sum_{i\in \mathbb{Z}}(2m-n-i)c_{n-m+i}^{a,b}{(n)}L_{n+i}\nonumber\\
  &&\ \ -\sum_{i\in \mathbb{Z}}(i+a+n+(b-1)m)d_{n-m+i}^{a,b}{(n)}I_{n+i}. \label{ll48}
\end{eqnarray}
Because $f$ is symmetric which implies  $f(L_m, L_n)= f(L_n, L_m)$, comparing both sides of the above equations, we obtain
\begin{eqnarray}
(i-n)a_i^{a,b}{(m)}&=&(n-i)c_i^{a,b}{(m)}\nonumber \\
&=&(2m-n-i)c_{n-m+i}^{a,b}{(n)}, \label{ac18}\\
(i+a+bn)b_i^{a,b}{(m)}&=&-(i+a+n+(b-1)m)d_{n-m+i}^{a,b}{(n)}\nonumber\\
&=&-(i+a+bn)d_i^{a,b}{(m)}. \label{ef118}
\end{eqnarray}

It is easy to see by Lemma \ref{refe} with (\ref{ac18}) that
\begin{equation}\label{tangxm0}
a_i^{a,b}{(m)}=c_i^{a,b}{(m)}=0, \ \text{ for all } m,i\in \mathbb{Z}.
\end{equation}
Next,  we shall finish the proof based on different cases of $a,b$.

{\it Case 1.} $a=0$ and $b\neq 0,1,2$. By (\ref{ef118}), $(i+bn)b_i^{0,b}{(m)}=-(i+bn)d_i^{0,b}{(m)}$, which implies from $b\neq 0$ that
$b_i^{0,b}{(m)}=-d_i^{0,b}{(m)}$  for any integers $m,i$.  By letting $m=0$ and $n=-i$ in (\ref{ef118}), we get $(b-1)nb_{-n}^{0,b}{(0)}=0,$
which from $b\neq 1$ gives that $b_{-n}^{0,b}{(0)}=0$ for any $n\neq 0$. Thus, $d_{-m+i}^{0,b}{(0)}=-b_{-m+i}^{0,b}{(0)}=0$ if $i\neq m$. This, together with (\ref{ef118}) by letting $n=0$, yields $ib_i^{0,b}{(m)}=0$ when $i\neq m$.  Therefore, we get $b_i^{0,b}{(m)}=0$ for all $i,m\in \mathbb{Z}$ with $i\neq 0,m$.
In particular, $d_{n-m}^{0,b}{(n)}=-b_{n-m}^{0,b}{(n)}=0$ if $n-m\neq 0, n$.  This, together with (\ref{ef118}) by letting $i=0$, yields $bnb_0^{0,b}{(m)}=0$
for all $n\neq 0,m$. The above discussion tells us that
\begin{equation}\label{tangxm1}
b_{i}^{0,b}{(m)}=d_{i}^{0,b}{(m)}=0, \ \text{ for all } i\neq m.
\end{equation}
Let $i=m$ in (\ref{ef118}), we get
\begin{equation}\label{tangxm2}
(m+bn)b_m^{0,b}{(m)}=(n+bm)d_{n}^{0,b}{(n)}, \ \text{ for all } m,n\in \mathbb{Z}.
\end{equation}
Take $n=0$ in the above equation, we obtain $mb_m^{0,b}{(m)}=bmb_{0}^{0,b}{(0)}$ for all $m\in \mathbb{Z}$. Denote $b_{0}^{0,b}{(0)}=\mu$, then
$b_m^{0,b}{(m)}=b\mu$ for all $m\in \mathbb{Z}$ with $m\neq 0$. From this, we have by (\ref{tangxm2}) that $b(b-1)(m-n)\mu=0$ for all $m,n\in \mathbb{Z}$ with
$m,n\neq 0$. It follows by $b\neq 0,1$ that $\mu=0$, and so $b_m^{0,b}{(m)}=0$ for all $m\in \mathbb{Z}$.  This, together with (\ref{tangxm1}), yields $b_{i}^{0,b}{(m)}=d_{i}^{0,b}{(m)}=0$ for all $i,m\in \mathbb{Z}$. With (\ref{tangxm0}), one can conclude by (\ref{aa14}) that
\begin{equation}\label{tangxm3}
\phi^{0,b}(L_n)=\psi^{0,b}(L_n)=0, \ \ \text{and} \  f(L_m, L_n)=0.
\end{equation}
This completes the proof of the claim in which $a=0$ and $b\neq 0,1,2$.

{\it Case 2.} $a\notin \mathbb{Z}$. In view of (\ref{ef118}), we see $(i+a)b_i^{a,b}{(m)}=-(i+a)d_i^{a,b}{(m)}$. Note that $i+a\neq 0$ since $a\notin \mathbb{Z}$, one has $b_i^{a,b}{(m)}=-d_i^{a,b}{(m)}$ for all $i\in \mathbb{Z}$. This, together with (\ref{ef118}), yields
\begin{eqnarray}
(i+a+bn)b_i^{a,b}{(m)}=(i+a+n+(b-1)m)b_{n-m+i}^{a,b}{(n)}.
\end{eqnarray}
Let $m=0$ in the above equation, then we have $(i+a) b_i^{a,b}{(m)}=(i+a+(b-1)m)t_{i-m}^{(0)}$. From this, by letting $k=i-m$ and $\mu_k=t_{k}^{(0)}\in \mathbb{C}$, we obtain $(m+k+a) b_{m+k}^{(m)}=(bm+k+a)\mu_k$, which implies
\begin{equation}\label{tangxm4}
 b_{m+k}^{(m)}=\frac{bm+k+a}{m+k+a}\mu_k=-d_{m+k}^{(m)}, \ \text{ for all } m,k\in \mathbb{Z}.
\end{equation}
Now, Equations  (\ref{tangxm0}), (\ref{tangxm4}) with (\ref{aa14}) yield that
\begin{eqnarray}
f(L_m,L_n)&=&\sum_{k\in \mathbb{Z}}(m+k+a+bn)\cdot \frac{bm+k+a}{m+k+a}\mu_k I_{n+m+k}\nonumber\\
&=&\sum_{k\in \mathbb{Z}} \left((b(m+n)+k+a)+\frac{b(b-1)mn}{m+k+a}\right) \mu_k I_{n+m+k}. \label{tangxm5}
\end{eqnarray}
Since $f(L_m,L_n)=f(L_m,L_n)$, so if there is some $\mu_k\neq 0$, then  by (\ref{tangxm5}) we get $b=0$ or $b=1$.
Note that the sequence $\Omega \doteq(\mu_k)_{  k\in \mathbb{Z}}$  which contains only finitely many nonzero entries for which the above equation makes sense.
This completes the proof of the claim for $a\notin \mathbb{Z}$. In addition, the above process also implies
\begin{equation} \label{ps=08}
\phi^{a,b}(L_n)=-\psi^{a,b}(L_n)=\sum_{k\in \mathbb{Z}} \frac{bm+k+a}{m+k+a}\mu_k I_{n+k}, \ \text{ for all } n\in \mathbb{Z}.
\end{equation}

\begin{claim}\label{claim 8}
$f(L_m,I_n)=0$ for all $m,n\in \mathbb{Z}$.
\end{claim}
By Lemma \ref{referee}, the linear maps $\phi^{a,b}, \psi^{a,b}: \mathcal{W}(a,b)\rightarrow \mathcal{W}(a,b)$ satisfy (\ref{aa15}) and (\ref{aa16}).
This, together with (\ref{dd10}), (\ref{eee1}), (\ref{ps=08}) or (\ref{tangxm3}),  yields that $f(L_m,I_n)$ and $f(I_n, L_m)$ are of the following forms respectively:
\begin{eqnarray}
\rho^{a,b}_1(L_m)I_n= \sum_{j\in \mathbb{Z}}(m-i)s_i^{a,b}{(n)}L_{m+i}-\sum_{i\in \mathbb{Z}}(i+a+bm)r_i^{a,b}{(n)}I_{m+i},\label{abc18} \\
\sum_{j\in \mathbb{Z}}(i-m)p_i^{a,b}{(n)}L_{i+m}+\sum_{i\in \mathbb{Z}}(i+a+bm)q_i^{a,b}{(n)}I_{m+i}=\theta^{a,b}_1(L_m)I_n.\label{abc28}
\end{eqnarray}
From (\ref{abc18}) we get $(m-i)s_i^{a,b}{(n)}=0$, which deduces $s_i^{a,b}{(n)}=0$ for all $i,n\in \mathbb{Z}$.
For any $m\neq n$, in view of (\ref{abc18}) in which $m=0$ we get $(i+a)r_i^{a,b}{(n)}=0$ for any $i\neq m$. This, together with $a\notin \mathbb{Z}$, yields $r_i^{a,b}{(n)}=0$ for any $n,i\in \mathbb{Z}$. Furthermore, it also follows that $\rho^{a,b}_1(L_m)=0$. Similarly, by (\ref{abc28}) one can obtain that
$p_i^{a,b}{(n)}=q_i^{a,b}{(n)}=\theta^{a,b}_1(L_n)=0$ for all $n,i\in \mathbb{Z}$. Therefore, we have shown that
\begin{equation}\label{tangxm8}
\phi^{a,b}(L_m)=\psi^{a,b}(I_m)=f(L_m,I_n)=0 \text{ for all } m,n\in \mathbb{Z}.
\end{equation}
The proof of the claim is completed.

\begin{claim}\label{claim 9}
$f(I_m,I_n)=0$ for all $m,n\in \mathbb{Z}$.
\end{claim}

Notice that (\ref{tangxm8}), the proof is similar to Claim \ref{claim 3}.

Finally, the proof of the proposition is completed by Claims \ref{claim 7}, \ref{claim 8} and \ref{claim 9}.
\end{proof}

Now we ready to give the proof of our main result.

{\bf The proof of Theorem \ref{posttheo}:}
The ``if" part is easy to verify, we now prove the ``only if" part.

Now we assume that $f$ is a biderivation of $\mathcal{W}(a,b)$. We shall finish the proof according to the cases of values allowed for $a,b$.

{\it Case 1.} $a\in \mathbb{Z}, b=0$.

By \cite{tangli2017} we know that if $f\in {\rm Bid}(\mathcal{W}(0,0))$, then there exist $\lambda \in \mathbb{C}$ and a sequence $\Omega=(\mu_k)_{  k\in \mathbb{Z}}$  which contains only finitely many nonzero entries such that $f=\lambda[x,y]+\Psi_\Omega(x,y)$ for all $x,y\in \mathcal{W}(0,0)$, where $\Psi_\Omega$ is given by Definition \ref{taa}.
Let $\sigma:\mathcal{W}(0,0)\rightarrow \mathcal{W}(a,0)$ be a linear map determined by $\sigma(L_m)= L_m$, $\sigma(I_m)=I_{m-a}$. Then by Lemma \ref{iso}, $\sigma$ is an isomorphism of Lie algebras.  Let  $f^\sigma$ be a linear map from $\mathcal{W}(a,0)$ into itself  determined by $f^\sigma(\sigma(x),\sigma(y))=\sigma(f(x,y))$ for all $x,y\in \mathcal{W}(a,0)$. Thanks to Lemma \ref{auto}, any biderivation of $\mathcal{W}(a,0)$ must be
of the form $f^\sigma$ which satisfies
\begin{eqnarray*}
f^\sigma (L_m,L_n)&=&f^\sigma(\sigma(L_m),\sigma(L_n)) \\
&=&\sigma(f(L_m,L_n))\\
&=&\lambda[L_m,L_n]+\sum_{k \in \mathbb{Z}}(m+n+k)\mu_k I_{m+n+k-a}\\
&=&\lambda[L_m,L_n]+\sum_{t \in \mathbb{Z}}(m+n+t+a)\mu_{t+a} I_{m+n+t}\\
&=& \lambda[L_m,L_n]+\Upsilon^a_{\Omega^\prime}(L_m,L_n),
\end{eqnarray*}
where $\Omega^\prime=(\mu_k^\prime)_{k\in \mathbb{Z}}$ with $\mu_k^\prime=\mu_{k+a}$; and obviously
$f^\sigma (L_m,I_n)=\lambda[L_m,I_n]$, $f^\sigma (I_n,L_m)=\lambda[I_n,L_m]$, $f^\sigma (I_m,I_n)=0$.
Hence, we have $f^\sigma(x,y)=\lambda [x,y]+\Psi_{\Omega^\prime}(x,y)$ for all $x,y\in \mathcal{W}(a,0)$.

{\it Case 2.} $a\notin \mathbb{Z}, b=0$.

By Lemma \ref{zhao}, ${\rm Bid}(\mathcal{W}(a,0) )={\rm Bid}_-(\mathcal{W}(a,0)) \oplus {\rm Bid}_+(\mathcal{W}(a,0)).$
By \cite{Hanw}, any $f_-\in {\rm Bid}_-(\mathcal{W}(a,0))$ has the form as $f_-(x,y)=\lambda [x,y]$ for all $x,y\in \mathcal{W}(a,0)$, where $\lambda\in \mathbb{C}$.
 Let $f_+\in {\rm Bid}_+(\mathcal{W}(a,0))$. Then by Proposition \ref{pro3}, there is a sequence $\Omega=(\mu_k)_{  k\in \mathbb{Z}}$  which contains only finitely many nonzero entries such that $f_+(x,y)=\Delta_\Omega(x,y)$ for all $x,y\in \mathcal{W}(a,0)$, where $\Delta_\Omega$ is given by Proposition \ref{pro3}. Therefore, we have $f_+(L_m,I_n)=f_+(I_m,L_n)=f_+(I_m,I_n)=0$ and
$$
f_+(L_m,L_n)= \sum_{k\in \mathbb{Z}}(k+a)\mu_k I_{m+n+k}=\sum_{k\in \mathbb{Z}}\mu^\prime_k I_{m+n+k}
$$
for all $m,n\in \mathbb{Z}$, where $\mu_k^\prime=(k+a)\mu_k.$ Let $\Omega^\prime=(\mu^\prime_k)_{k\in \mathbb{Z}}$, we see that $f_+=\Psi_{\Omega^\prime}$.
Hence, in this case any biderivation $f$ of $\mathcal{W}(a,0)$ is of the form
$$
f(x,y)=f_-(x,y)+f_+(x,y)=\lambda [x,y]+\Psi_{\Omega^\prime}(x,y), \text{ for all } x,y\in \mathcal{W}(a,0).
$$

{\it Case 3.} $a\in \mathbb{Z}, b=1$.

By Proposition \ref{pro1}, we know that if $f_+\in {\rm Bid}_+(\mathcal{W}(0,1))$ then
there is a sequence $\Omega=(\mu_k)_{  k\in \mathbb{Z}}$  which contains only finitely many nonzero entries such that
$f_+(x,y)=\Upsilon_\Omega^{0}(x,y)$ for all $x,y\in \mathcal{W}(0,1)$, where $\Upsilon_\Omega^0$ is given by Definition \ref{taa}.
Let $\sigma:\mathcal{W}(0,1)\rightarrow \mathcal{W}(a,1)$ be a linear map determined by $\sigma(L_m)= L_m$, $\sigma(I_m)=I_{m-a}$. Then by Lemma \ref{iso}, $\sigma$ is an isomorphism of Lie algebras.  Let  $f_+^\sigma$ be a linear map from $\mathcal{W}(a,1)$ into itself  determined by $f_+^\sigma(\sigma(x),\sigma(y))=\sigma(f_+(x,y))$ for all $x,y\in \mathcal{W}(a,1)$. Thanks to Lemma \ref{auto}, any symmetric biderivation of $\mathcal{W}(a,0)$ must be
of the form $f_+^\sigma$ which is given by
\begin{eqnarray*}
f_+^\sigma (L_m,L_n)&=&f_+^\sigma(\sigma(L_m),\sigma(L_n))\\
&=&\sigma(f_+(L_m,L_n))\\
&=& \sum_{k \in \mathbb{Z}}(m+n+k)\mu_k I_{m+n+k-a}\\
&=& \sum_{t \in \mathbb{Z}}(m+n+t+a)\mu_{t+a} I_{m+n+t}\\
&=&\Upsilon^a_{\Omega^\prime}(L_m,L_n),
\end{eqnarray*}
where $\Omega^\prime=\{\mu_k^\prime=\mu_{k+a}|k\in \mathbb{Z}\}$; and obviously $f_+^\sigma (L_m,I_n)=f_+^\sigma (I_n,L_m)=f_+^\sigma (I_m,I_n)=0$.
On the other hand, if $f_-\in {\rm Bid}_-(\mathcal{W}(a,1))$, then by \cite{Hanw} we see that $f_-(x,y)=\lambda [x, y]$ for some $\lambda\in \mathbb{C}$. Now, by Lemma \ref{zhao} we conclude that any biderivation of $\mathcal{W}(a,1)$ is of the form
$$
f_-(x,y)+f_+^\sigma(x,y)=\lambda [x, y]+\Upsilon^a_{\Omega^\prime}(x,y),\ \text{for all } x,y\in \mathcal{W}(a,1).
$$

{\it Case 4.} $a\notin \mathbb{Z}, b=1$.

For any  $f_-\in {\rm Bid}_-(\mathcal{W}(a,1))$, by \cite{Hanw} one has $f_-(x,y)=\lambda [x, y]$ for some $\lambda\in \mathbb{C}$.
On the other hand, if $f_+\in {\rm Bid}_+(\mathcal{W}(a,1))$ then by Proposition \ref{pro3}, there is a sequence $\Omega=(\mu_k)_{  k\in \mathbb{Z}}$  which contains only finitely many nonzero entries such that $f_+(x,y)=\Delta_\Omega(x,y)$ for all $x,y\in \mathcal{W}(a,1)$, where $\Delta_\Omega$ is given by Proposition \ref{pro3}. Therefore, we have $f_+(L_m,I_n)=f_+(I_m,L_n)=f_+(I_m,I_n)=0$ and
$$
f_+(L_m,L_n)= \sum_{k\in \mathbb{Z}}(m+n+k+a)\mu_k I_{m+n+k}=\Psi_\Omega(L_m,L_n)
$$
for all $m,n\in \mathbb{Z}$. Then we see that $f_+=\Psi_{\Omega}$.
Thanks to Lemma \ref{zhao}, we deduce that any biderivation of $\mathcal{W}(a,1)$ is of the form
$
f_-(x,y)+f_+(x,y)=\lambda [x, y]+\Psi_{\Omega}(x,y)
$
for all $x,y\in \mathcal{W}(a,1)$.

{\it Case 5.} $a\in \mathbb{Z}, b=-1$.

By \cite{tang2017}, we see that if $f\in {\rm Bid}(\mathcal{W}(0,-1))$, then there exist $\lambda,\mu \in \mathbb{C}$ such that
$f=\lambda[x,y]+\Theta^0_\mu(x,y)$ for all $x,y\in \mathcal{W}(0,-1)$, where $\Theta^0_\mu$ is given by Definition \ref{taa}.
Let $\sigma:\mathcal{W}(0,-1)\rightarrow \mathcal{W}(a,-1)$ be a linear map determined by $\sigma(L_m)= L_m$, $\sigma(I_m)=I_{m-a}$. Then by Lemma \ref{iso}, $\sigma$ is an isomorphism of Lie algebras.  Let  $f^\sigma$ be a linear map from $\mathcal{W}(a,-1)$ into itself  determined by $f^\sigma(\sigma(x),\sigma(y))=\sigma(f(x,y))$ for all $x,y\in \mathcal{W}(a,-1)$. Thanks to Lemma \ref{auto}, any biderivation of $\mathcal{W}(a,-1)$ must be
of the form $f^\sigma$ which satisfies
\begin{eqnarray*}
f^\sigma (L_m,L_n)&=&f^\sigma(\sigma(L_m),\sigma(L_n))\\
&=&\sigma(f(L_m,L_n))\\
&=&\sigma(\lambda[L_m,L_n]+\Theta^0_\mu(L_m,L_n))\\
&=& \lambda[L_m,L_n]+\sum_{k \in \mathbb{Z}}\mu_k \sigma(I_{m+n})\\
&=& \lambda[L_m,L_n]+\sum_{k \in \mathbb{Z}}\mu_k I_{m+n-a}\\
&=&\lambda[L_m,L_n]+\Theta^a_\mu(L_m,L_n);
\end{eqnarray*}
and obviously
$f^\sigma (L_m,I_n)=\lambda[L_m,I_n], f^\sigma (I_n,L_m)=\lambda[I_n,L_m]$, $f^\sigma (I_m,I_n)=0$.
Hence, we have $f^\sigma(x,y)=\lambda [x,y]+\Theta^a_\mu(x,y)$ for all $x,y\in \mathcal{W}(a,-1)$.

{\it Case 6.} (i) $a\in \mathbb{Z}, b=2$; (ii) $a\in \mathbb{Z}, b\notin \{-1, 0, 1, 2\}$ and (iii) $a\notin \mathbb{Z}, b\notin\{0,1\}$.

Proposition \ref{pro2} tells us that if $f_+\in {\rm Bid}_+(\mathcal{W}(0,2))$ then $f_+=0$.  In a similar way to the proof of Case 1, by Lemmas \ref{iso} and \ref{auto} we conclude that any $f_+^\sigma\in {\rm Bid}_+(\mathcal{W}(a,2))$ has to be $0$ for all $a\in \mathbb{Z}$. In view of Proposition \ref{pro3}, we see that whatever $a\in \mathbb{Z}, b\notin \{-1, 0, 1, 2\}$ or $a\notin \mathbb{Z}, b\notin\{0,1\}$, any $f_+\in {\rm Bid}_+(\mathcal{W}(a,b))$ has to be $0$.
This indicated that ${\rm Bid}_+(\mathcal{W}(a,b))$ contains only zero biderivation in all cases (i),(ii) and (iii). On the other hand, it is easy to see by \cite{Hanw} that ${\rm Bid}_-(\mathcal{W}(a,b))$ contains only inner biderivation in these cases. Therefore, by Lemma \ref{zhao} we deduce that any biderivation $f$ of $\mathcal{W}(a,b)$ must be inner in all cases (i),(ii) and (iii).

Now, we summarize Cases 1-6 and complete the proof of Theorem \ref{posttheo}.

\section{Post-Lie algebra structures on $\mathcal{W}(a,b)$}\label{sec4}

Recall that the post-Lie algebras have been introduced by Valette in connection with the homology of partition
posets and the study of Koszul operads \cite{vela}. As \cite{Burde1} pointed out, post-Lie algebras are natural common generalization of pre-Lie algebras  and LR-algebras in the geometric context of nil-affine actions of Lie groups. Recently, many authors study some post-Lie algebras and post-Lie algebra structures  \cite{Burde2,Burde1,Mun,pan,tang2014}. In particular, the authors \cite{Burde1} study the commutative post-Lie algebra structure on Lie algebra.  By using our results, we can characterize the  commutative post-Lie algebra structure on $\mathcal{W}(a,b)$.
Let us recall the following definition of commutative post-Lie algebra.

\begin{definition}\label{post} \cite{tang2017}
Let $(L, [, ])$ be a complex Lie algebra. A commutative post-Lie algebra structure on $L$ is a
$\mathbb{C}$-bilinear product $x\circ y$ on $L$ satisfying the following identities:
\begin{eqnarray*}
&& x \circ y = y\circ x, \nonumber\\
&& [x, y] \circ z =  x \circ (y \circ z)-y \circ (x \circ z), \label{post6}\\
&& x\circ [y, z] = [x\circ y, z] + [y, x \circ z] \nonumber
\end{eqnarray*}
for all $x, y, z \in L$. We also say $(L, [, ], \circ)$ a commutative post-Lie algebra.
\end{definition}

A post-Lie algebra $(L, [, ], \circ)$ is said to be trivial if $x\circ y=0$ for all $x,y\in L$.
The following lemma shows the connection between commutative post-Lie algebra and biderivation of a Lie algebra.

\begin{lemma}\label{postbide} \cite{tang2017}
Suppose that $(L, [, ], \circ)$ is a commutative post-Lie algebra. If we define a bilinear map $f : L\times L \rightarrow L$ by $f(x,y)=x\circ y$ for all $x,y\in L$, then $f$ is a biderivation of $L$.
\end{lemma}

\begin{theorem}\label{refreesttt}
Any commutative post-Lie algebra structure on the algebra $\mathcal{W}(a,b)$ is trivial.
\end{theorem}

\begin{proof}
Suppose that $(\mathcal{W}(a,b), [, ], \circ)$ is a commutative post-Lie algebra. By Lemma \ref{postbide} and Theorem  \ref{posttheo}, we know that there are $\lambda,\mu \in \mathbb{C}$ and a sequence $\Omega=(\mu_k)_{  k\in \mathbb{Z}}$  which contains only finitely many nonzero entries such that
$$
x\circ y=\begin{cases}
\lambda[x,y]+\Psi_\Omega(x,y), & \text{if} \ b=0, \\
\lambda[x,y]+\Upsilon^a_\Omega(x,y), & \text{if} \ b=1, \\
\lambda[x,y]+\Theta^a_\mu(x,y), & \text{if} \ a\in \mathbb{Z}, b=-1, \\
\lambda[x,y], & \text{otherwise}. \\
\end{cases}
$$ for all $x,y\in \mathcal{H}$, where $\Psi_\Omega,\Upsilon^a_\Omega$ and $\Theta^a_\mu$ are given by Definition \ref{taa}. Because the product $\circ$ is commutative, we have by $L_1\ \circ L_2=L_1\ \circ L_2$ that $\lambda=\mu=0$.  By (\ref{post6}), we see that
$$
[L_m,L_n]\circ L_t=L_m\circ(L_n\circ L_t)-L_n\circ(L_m\circ L_t)
$$
for all $m,n,t\in \mathbb{Z}$. If there is $\mu_k\in \Omega$ such that $\mu_k\neq 0$, then it is easy to find some $m,n,t\in \mathbb{Z}$ such that the left-hand side of the above equation contains at least a nonzero item, whereas the right-hand side is equal to zero, which is a contradiction. Thus, we have $\mu_i=0$ for any $i\in \mathbb{Z}$. In other words, $x\circ y=0$ for all $x,y\in \mathcal{W}(a,b)$.
\end{proof}

\section*{ Acknowledgements}
The author  thanks  Dr.  C. Xia for sending the paper \cite{Hanw}, and thanks
Prof. K. Zhao for some suggestions to revise the paper. The author was partly supported  by NSF of China (grant no. 11771069) and NSF of Heilongjiang Province (Grant A2015007).

\end{document}